\documentclass[reqno]{amsart}
\usepackage{amsmath,amsthm,amsfonts, amssymb,amscd}
\usepackage[all]{xy}

\numberwithin{equation}{section}

\newtheorem{theorem}{Theorem}[section]
\newtheorem{lemma}[theorem]{Lemma}
\newtheorem{proposition}[theorem]{Proposition}
\theoremstyle{definition}
\newtheorem{definition}[theorem]{Definition}

\newtheorem{example}[theorem]{Example}

\newcommand{\lex}{\,\overrightarrow{\times}\,}

\newcommand{\Rad}{\mbox{\rm Rad}}

\begin{document}
\title[$n$-dimensional Observables on $k$-Perfect MV-Algebras and $k$-Perfect Effect Algebras. I.]{$n$-dimensional Observables on $k$-Perfect MV-Algebras and $k$-Perfect Effect Algebras. I. Characteristic Points}
\author[A. Dvure\v{c}enskij, D. Lachman]{Anatolij Dvure\v{c}enskij$^{1,2}$, Dominik Lachman$^2$}
\maketitle

\begin{center}  \footnote{Keywords: $n$-dimensional observable; $n$-dimensional spectral resolution; characteristic point; unital po-group; interpolation; $k$-perfect MV-algebra; lexicographic MV-algebra; $k$-perfect effect algebra

 AMS classification: 06D35, 06F20, 81P10

The first author acknowledges the support by
the Slovak Research and Development Agency under contract APVV-16-0073 and the grant VEGA No. 2/0142/20 SAV, and the second author acknowledges the support by the Austrian Science Fund (FWF): project I 4579-N and the Czech Science Foundation (GA\v CR): project 20-09869L.}
Mathematical Institute,  Slovak Academy of Sciences\\
\v Stef\'anikova 49, SK-814 73 Bratislava, Slovakia\\
$^2$ Depart. Algebra  Geom.,  Palack\'{y} Univer.\\
17. listopadu 12, CZ-771 46 Olomouc, Czech Republic\\

E-mail: {\tt
dvurecen@mat.savba.sk,\quad dominiklachman@seznam.cz}
\date{}%
\end{center}

\begin{abstract}
In the paper, we investigate a one-to-one correspondence between $n$-dimensional observables and $n$-dimensional spectral resolutions with values in a kind of a lexicographic form of quantum structures like perfect MV-algebras or perfect effect algebras. The multidimensional version of this problem is more complicated than a one-dimensional one because if our algebraic structure is $k$-perfect for $k>1$, then even for the two-dimensional case we have more characteristic points. The obtained results are also applied to existence of an $n$-dimensional meet joint observable of $n$ one-dimensional observables on a perfect MV-algebra.
The results are divided into two parts. In Part I, we present notions of $n$-dimensional observables and $n$-dimensional spectral resolutions with accent on lexicographic type effect algebras and lexicographic MV-algebras. We concentrate on characteristic points of spectral resolutions and the main body is in Part II where one-to-one relations between observables and spectral resolutions are presented.
\end{abstract}

\section{Introduction}

Mathematical foundations of quantum mechanics are going back to the thirties when a seminal paper \cite{BiNe} has appeared in which it was shown that events that can be observed during quantum mechanical measurement do not fulfil axioms of Boolean algebras and rather of a more complicated structure which we call today a quantum structure. The basic models of quantum structures are  Boolean algebras, orthomodular lattices and posets, orthoalgebras, and since the beginning of the nineties also effect algebras presented in \cite{FoBe}. The latter ones are very successful structures because they combine both sharp and unsharp (fuzzy) features of quantum measurements and in many situations they are connected with po-groups and even with lattice ordered groups. An orthodox example of effect algebras is the system $\mathcal E(H)$ of all Hermitian operators of a Hilbert space $H$ that are between the zero and the identity operators. They form an interval in the unital po-group $\mathcal B(H)$ of all Hermitian operators of $H$. Another important structure connected with a Hilbert space is the system $\mathcal P(H)$ of all orthogonal projectors on $H$. These effect algebras contain also MV-effect algebras, which are equivalent to MV-algebras that describe the \L ukasiewicz infinite valued logic, \cite{Cha}. MV-algebras play an analogous role in effect algebras as Boolean algebras do in orthomodular lattices.

A special kind of MV-algebras are so-called perfect MV-algebras, i.e. MV-algebras where each element either belongs to its radical or to its co-radical, and they will be in the center of our investigation. Perfect MV-algebras have an important property because the Lindenbaum algebra of the first order \L ukasiewicz logic is not semisimple, and the valid but unprovable formulas are precisely the formulas whose negations determine the radical of the Lindenbaum algebra, i.e. the co-infinitesimal elements of such an  algebra, see \cite{DiGr}.

If we measure some quantity in a classical system, we use a model of a measurable space $(\Omega,\mathcal S)$, where $\mathcal S$ is a $\sigma$-algebra of subsets of a set $\Omega \ne \emptyset$. The measurement is performed by a measurable function $f: \Omega \to \mathbb R$. The measurability of $f$ means that $f^{-1}(A) \in \mathcal S$ for each Borel sets $A\in \mathcal B(\mathbb R)$. Then the mapping $x_f:\mathcal B(\mathbb R) \to \mathcal S$ defined by $x_f(A)=f^{-1}(A)$, $A \in \mathcal B(\mathbb R)$, is a kind of a $\sigma$-homomorphism of effect algebras. We note that an observable in the case of $\mathcal P(H)$ is an orthogonal projector-valued measure and in the case of $\mathcal E(H)$ an observable is a positive operator-valued measure, see e.g. \cite{DvPu}.

If $x$ is an observable, then the mapping $F_x(t)=x((-\infty,t))$, $t\in \mathbb R$, is (i) monotone, (ii) left continuous, (iii) going to $0$ and $1$ when $t$ is going to $-\infty$ and to $+\infty$, respectively. The mapping $F_x$ is said to be a spectral resolution corresponding to an observable $x$. If a mapping $F:\mathbb R \to M$ satisfies (i)--(iii), the question is whether there exists an observable $x$ on a $\sigma$-complete MV-algebra $M$ such that $F= F_x$. The positive answer was given in \cite{DvKu} also for spectral resolutions on monotone Dedekind $\sigma$-complete effect algebras with RDP. We note that one of the first studies which shows a one-to-one correspondence between observables in $\mathcal P(H)$ and spectral resolutions was established in \cite{Cat} for quantum logics.
This one-to-one correspondence between observables and spectral resolutions was extended for $\Rad$-Dedekind $\sigma$-complete perfect MV-algebras in \cite{DDL}, however it is important to note that the notion of a spectral resolution was necessary to strengthen. This one-to-one relation was established also for $k$-perfect MV-algebras in \cite{DvLa} and for general lexicographic MV-algebras and lexicographic effect algebras in \cite{DvLa1,DvLa2}.

If we simultaneously measure $n$ quantities, $n\ge 1$, then we have an $n$-dimensional random vector $T:\Omega \to \mathbb R^n$ such that $T^{-1}(A) \in \mathcal S$ for each $n$-dimensional Borel set $A \in \mathcal B(\mathbb R^n)$ and again $x_T(A):=T^{-1}(A)$, $A \in \mathcal B(\mathbb R^n)$, is a kind of a $\sigma$-homomorphism. Therefore, in the paper we will study $n$-dimensional observables $x$ and $n$-dimensional spectral resolutions as mappings $F$ from $\mathbb R^n$ into the algebraic structure which are monotone, left-continuous with non-negative increments, going to $0$ if one variable goes to $-\infty$ and going to $1$ if all variables go to $+\infty$. The one-to-one correspondence between $n$-dimensional observables and $n$-dimensional spectral resolutions have been established in \cite{DvLa3} for $\sigma$-complete MV-algebras and for monotone $\sigma$-complete effect algebras with RDP.

The principal aims of the present contribution are:

(1) To establish a one-to-one correspondence between $n$-dimensional observables and $n$-dimensional spectral resolutions for perfect and $k$-perfect MV-algebras and $k$-perfect effect algebras.

(2) For spectral resolutions $F$ on lexicographic type algebras, there appear characteristic points of $F$ which for case of $\sigma$-complete MV-algebras or Dedekind monotone $\sigma$-complete effect algebras do not exist.

(3) As for one-dimensional observables, it was necessary to strengthen the notion of an $n$-dimensional spectral resolutions. In addition, if $n\ge 2$ and $k>1$, there appeared a new phenomenon that if $M_i=\{(i,a)\colon  (i,a)\in \Gamma(\mathbb Z \lex G,(k,0))\}$, $i=1,\ldots,k$, and $T_i=\{(s_1,\ldots,s_n)\in \mathbb R^n\colon F(s_1,\ldots,s_n) \in M_i\}$, then it can happen that $T_i$ defines more characteristic points which if $n=1$ or $k=1$ is not possible.

(4) To describe characteristic points at least for $n=2$. That is, to show that every $T_i$ describes only finitely many characteristic points. For higher dimensions such a description is unknown. The results are then applied to existence of a kind of a joint observable of $n$ one-dimensional observables and to show how we can define a sum of $n$-dimensional observables on perfect MV-algebras.

The paper is organized as follows. Part I. Section 2 gathers the basic facts on effect algebras, MV-algebras, and their $k$-perfect forms. Section 3 describes $n$-dimensional observables, their $n$-dimensional spectral resolutions on perfect MV-algebras, and it shows some of their basic properties. Section 4 studies characteristic points for $k$-perfect MV-algebras. We show that in the two-dimensional case, every two-dimensional spectral resolution defines only finitely many characteristic points, and some illustrating examples are present, too.

Part II starts with Section 5 where we strengthened the definition of an $n$-dimensional spectral resolution which will hold for lexicographic MV-algebras and lexicographic effect algebras. For $n=2$, we establish a one-to-one correspondence between spectral resolutions and observables for perfect MV-algebras and perfect effect algebras. In Section 6, we study some limit properties of spectral resolutions and extend the notion of an $n$-dimensional spectral resolution also for limits. Section 7 shows that every $n$-dimensional spectral resolution on a perfect MV-algebra and on a perfect effect algebra can be uniquely extended to an $n$-dimensional observable. This extension for $n$-dimensional spectral resolutions on $k$-perfect MV-algebras and $k$-perfect effect algebras is presented for spectral resolutions with the ordering property. An application of the gained results shows that for $n$ one-dimensional observables there is a some kind of a joint $n$-dimensional observable, see Section 9.

\section{$k$-perfect MV-algebras and $k$-perfect Effect Algebras}

Effect algebras were introduced in \cite{FoBe} as follows: We say that an {\it effect algebra} is a partial algebra $E = (E;+,0,1)$ with a partially defined operation $+$ and with two constant elements $0$ and $1$ such that, for all $a,b,c \in E$,
\begin{enumerate}
\item[(i)] $a+b$ is defined in $E$ iff $b+a$ is defined, and in
which  case $a+b = b+a$;

\item[(ii)] $a+b, (a+b)+c$ are defined iff $b+c$ and $a+(b+c)$ are
defined, and in which case $(a+b)+c = a+(b+c)$;

\item[(iii)] for any $a \in E$, there exists a unique element $a'
\in E$ such that $a+a'=1$;

\item[(iv)] if $a+1$ is defined in $E$, then $a=0$.
\end{enumerate}

The partial operation $+$ yields a partial ordering $\le$ on $E$ defined by $a\le b$ iff there is $c \in E$ such that $a+c=b$; we write $c=b-a$. Then $0\le a \le 1$ for each $a \in E$. If $E$ under $\le$ is a lattice, we call $E$ a {\it lattice effect algebra}. A lattice effect algebra $E$ is an {\it MV-effect algebra} if $a \wedge b = 0$ implies $a + b$ is defined in $E$ for $a,b \in E$. It is well-known that MV-effect algebras are arising from MV-algebras, see \cite[Thm 1.8.12]{DvPu}. We recall that an MV-{\it algebra} is an algebra $(M;\oplus,',0,1)$ (henceforth written simply as $M=(M;\oplus,',0,1)$) of type $(2,1,0,0)$, where $(M;\oplus,0)$ is a commutative monoid with the neutral element $0$ and for all $a,b\in M$, such that we have:
\begin{enumerate}
	\item[(i)] $a''=a$;
	\item[(ii)] $a\oplus 1=1$;
	\item[(iii)] $a\oplus (a\oplus b')'=b\oplus (b\oplus a')'$;
    \item[(iv)] $0'=1$.
\end{enumerate}
\noindent
In any MV-algebra $(M;\oplus,',0,1)$, we can also define the following term operation:
\[
a\odot b:=(a'\oplus b')'.
\]
\noindent

If we define a partial operation $+$ on an MV-algebra $M$ by $a+b$ is defined iff $a\le b'$ (equivalently $a\odot b = 0$), and in such a case, we put $a+b=a\oplus b$. Then $(M;+,0,1)$ is exactly an MV-effect algebra. Conversely, if $E$ is an MV-effect algebra, then there is an MV-algebra $M$ such that $(M;+,0,1)\cong (E;+,0,1)$. For more info about effect algebras see \cite{FoBe,DvPu} and about MV-algebras see \cite{CDM}.

Effect algebras and specially MV-algebras are connected with Abelian partially ordered groups. We say that an Abelian group $(G;+,0)$ endowed with a partial ordering $\le$ is a {\it partially ordered group} (po-group, in abbreviation) if $g\le h$ implies $g+k\le h+k$ for each $k\in G$. If the partial order is a lattice order, then $G$ is said to be an $\ell$-group. An element $u$ of a po-group $G$ is a {\it strong unit} if, for each $g\in G$, there is an integer $n\ge 1$ such that $g\le nu$. A couple $(G,u)$, where $G$ is a po-group with a fixed strong unit, is said to be a {\it unital group}. The positive (negative) cone of a po-group $G$ is the set $G^+=\{g \in G \colon g\ge 0\}$ ($G^-=\{g \in G\colon g\le 0\}$).

If $(G,u)$ is a unital $\ell$-group, then $\Gamma(G,u)=([0,u];\oplus,',0,u)$, where $[0,u]=\{g\in G\colon 0\le g \le u\}$, $a\oplus b= (a+b)\wedge u$, $a'=u-a$, is a prototypical example of an MV-algebra as it follows from the Mundici's representation result, see \cite{Mun}. If $(G,u)$ is a unital po-group, then $\Gamma_{ea}(G,u)=([0,u];+,0,1)$, where $a+b$ is the sum of $a$ and $b$ in $G$ if the sum is in the interval $[0,u]$, is a so-called {\it interval effect algebra}. Hence, $\Gamma_{ea}(G,u)$ is an MV-effect algebra iff $(G,u)$ is a unital $\ell$-group, and an effect algebra $E$ is an MV-effect algebra iff $E\cong \Gamma_{ea}(G,u)$ for some unital $\ell$-group $(G,u)$, compare \cite[Thm 1.8.12]{DvPu}.

A po-group $G$ is with {\it interpolation} if $g_1,g_2\le h_1,h_2$ entails an element $g\in G$ such that $g_1,g_2\le g \le h_1,h_2$. Every $\ell$-group is with interpolation. An effect algebra satisfies the {\it Riesz Decomposition Property} (RDP for short) if $a_1 + a_2 = b_1 + b_2$ implies there exist four elements $c_{11}, c_{12}, c_{21}, c_{22} \in E$ such that $a_1 = c_{11} + c_{12},$ $a_2 = c_{21} + c_{22}$, $b_1 = c_{11} + c_{21}$, and $b_2= c_{12} + c_{22}$. The basic result of Ravindran says that there is a one-to-one representation between effect algebras with RDP and unital po-groups with interpolation, see \cite{Rav}.

Let $H$ and $G$ be two po-groups. Then $H\lex G$ is the {\it lexicographic product} of $H$ with $G$  if on the direct product $H\times G$ we introduce the lexicographic order: $(h_1,g_1)\le (h_2,g_2)$ iff either $h_1<h_2$ or $h_1= h_2$ and $g_1\le g_2$. Then (1) $H\lex G$ is an interpolation group iff both $H$ and $G$ are interpolation groups and either $H$ satisfies strict interpolation or $G$ is directed, \cite[Cor 2.12]{Goo}, (2) $H\lex G$ is an $\ell$-group if $H$ is linearly ordered and $G$ is an $\ell$-group, \cite[(d) p. 26]{Fuc}. For more info about po-groups, we recommend to consult with \cite{Goo, Fuc}. We note that all groups used in the paper will be Abelian written additively.

We say that a poset $G$ (in particular a po-group $G$) is {\it monotone $\sigma$-complete} (or {\it Dedekind monotone $\sigma$-complete}) provided that every ascending (descending) sequence $x_1\le x_2\le \cdots$ ($x_1\ge x_2 \ge \cdots$) in $G$ which is bounded above (below) in $G$ has a supremum (infimum) in $G$.

Let $E$ be an effect algebra. Given finitely many elements $a_1,\ldots,a_n\in E$, associativity of $+$ allows us to define their sum $a= a_1+\cdots + a_n$ unambiguously if it exists; in such a case, elements $a_1,\ldots, a_n$ are said to be {\it summable} with the sum $a=a_1+\cdots+a_n:= \sum_{i=1}^na_i$. Clearly, if $i_1,\ldots,i_n$ is any permutation of $1,\ldots,n$, then $\sum_{i=1}^na_i=\sum_{j=1}^na_{i_j}$.
A sequence $(a_n)_n$ of elements of an effect algebra $E$ is {\it summable} if every finite subsystem of $(a_n)_n$ is summable. If the element
$$a=\bigvee\{\sum_{i\in F}a_i\colon \text{ $F$ is any finite subset of } \mathbb N\}
$$
exists in $E$, $a$ is said to be the {\it sum} of $(a_n)_n$, and we write $a=\sum_{n=1}^\infty a_n$. We note that $a$ does not depend on the ordering of $(a_n)_n$, and if a system is summable, then it does not mean automatically that it has a sum. Indeed, let $E=\Gamma_{ea}(\mathbb Z\lex \mathbb Z,(1,0))$, then the system $(a_n)_n$, where $a_n=(0,1)$ for each $n\ge 1$, is summable but its sum does not exist in $E$.

In the center of our investigation, we investigate algebraic structures which are of the form $E=\Gamma_{ea}(H\lex G,(u,0))$, where $(H,u)$ is a unital po-group and $G$ is a directed monotone $\sigma$-complete po-group with interpolation or they are lexicographic MV-algebras, see \cite{DFL}, i.e. ones of the form $M=\Gamma(H\lex G,(u,0))$, where $(H,u)$ is a linear unital po-group and $G$ is a Dedekind $\sigma$-complete $\ell$-group. For each $h\in [0,u]_H:=\{h \in H\colon 0\le h \le u\}$, we denote by $E_h$ the set of elements of $E$, whose first coordinate is $h$. Clearly, $E= \bigcup\{E_h\colon h \in [0,u]_H\}$ is a disjoint union of all $E_h$ with $h \in [0,u]_H$, and $E_h\ne \emptyset$. We can also write $E=(E_h\colon h \in [0,u]_H)$. Clearly if $h_1,h_2 \in [0,u]_H$, $h_1<h_2$, then $E_{h_1}\le E_{h_2}$, i.e. each element of $E_{h_1}$ is less than any element of $E_{h_2}$. As posets, the set $E_0$ is isomorphic to $G^+$, $E_u$ is isomorphic to $G^-$ and all other $E_h$'s are isomorphic to $G$. Analogously we write $M=(M_h\colon h \in [0,h]_H)$.

Now, we will concentrate to MV-algebras. Given $x\in M$, we put
$$
0x:=0,\quad 1x=x, \quad (n+1)x= nx +x, \text{ if $nx +x$ exists in } M,\ n\ge 1.
$$
We denote by $\Rad(M)$ the {\it radical} of $M$, i.e. the set of $x\in M$ such that $nx$ exists in $M$ for each $n\ge 1$. An MV-algebra is {\it perfect} if every element $x\in M$ is either from $\Rad(M)$ or is from its co-radical $\Rad(M)'=\{a'\colon a\in \Rad(M)\}$. An MV-algebra $M$ is {\it $\Rad$-Dedekind $\sigma$-complete} if the radical $\Rad(M)$ is a Dedekind $\sigma$-complete poset, i.e. if $(a_n)_n$ is a a sequence of elements from $\Rad(M)$ which is bounded from above by an element $b\in \Rad(M)$, then $\bigvee_n a_n$ exists in $M$ and it belongs to $\Rad(M)$. According to \cite{DiLe}, an MV-algebra $M$ is perfect iff there is an $\ell$-group $G$ such that $M \cong \Gamma(\mathbb Z \lex G,(1,0))$, where $\mathbb Z$ is the group of integers. Therefore, we say that an MV-algebra $M$ is $k$-perfect, where $k\ge 1$ is an integer, if $M\cong \Gamma(\mathbb Z \lex G,(k,0))$. Clearly a $1$-perfect MV-algebra is simply a perfect MV-algebra. In addition, a    $k$-perfect MV-algebra $M$ is $\Rad$-Dedekind $\sigma$-complete iff $M \cong \Gamma(\mathbb Z \lex G,(k,0))$, where $G$ is a Dedekind $\sigma$-complete $\ell$-group, see \cite[Thm 2.5]{DvLa}.

In this analogy, if $G$ is a directed monotone $\sigma$-complete po-group and $k\ge 1$ is an integer, then every effect algebra isomorphic to $\Gamma_{ea}(\mathbb Z\lex G,(k,0))$ is said to be a $k$-{\it perfect effect algebra}. Often we need that $G$ is with interpolation and Dedekind monotone $\sigma$-complete. Such effect algebras for $k=1$ were studied in \cite{Dvu1}.

The following simple but useful statements were established in \cite{DvLa1}, \cite[Lem 3.2]{DvLa2}.

\begin{lemma}\label{le:2.1}
Let $M=\Gamma(H\lex G,(u,0))$, where $(H,u)$ is a linearly ordered unital group and $G$ is a Dedekind $\sigma$-complete $\ell$-group.

{\rm (1)} Let $(x_n)_n$ be a monotone sequence of elements from
$M$. Then $\bigvee_n x_n$ $(\bigwedge_n x_n)$
exists and belongs to $M_h$ for some $h \in [0,u]_H$ if and only if there is
some upper (lower) bound $x\in M_h$ of $(x_n)_n$ and there is some
$x_n\in M_h$.

{\rm (2)} Given $h \in [0,u]_H$, if $(a_n)_n$ is a sequence of elements from $M_h$ which is bounded above (below) by some element $a\in M_h$, then $\bigvee_na_n \in M_h$ $(\bigwedge_a a_n\in M_h)$.

{\rm (3)} Let $(a_n)_n$ be a summable finite or infinite sequence of elements of $M$. Then all but finitely many $a_n$'s belong to $M_0$ and every subsequence $(a_{n_i})_i$ of $(a_n)_n$ is summable with sum in $M$.
\end{lemma}

An analogous statement holds also for $E=\Gamma_{ea}(H\lex G,(u,1))$, where $(H,u)$ is a unital po-group with interpolation and $G$ is a directed Dedekind monotone $\sigma$-complete po-group with interpolation.

\section{$n$-dimensional Observables and Perfect MV-algebras}

In the section, we define an $n$-dimensional observable and a corresponding $n$-dimensional spectral resolution on a lexicographic MV-algebra. We show how characteristic points can be described, and we present their basic properties.

The following definition of an $n$-dimensional observable for lexicographic MV-algebras coincides with one for $n$-dimensional observables on $\sigma$-complete MV-algebras, see \cite{DvLa}.

\begin{definition}\label{de:3.1}
{\rm Let $n\ge 1$ be an integer and let $M=\Gamma(H\lex G,(u,0))$, where $(H,u)$ is a linearly ordered unital group and $G$ is a Dedekind $\sigma$-complete $\ell$-group. A mapping $x:\mathcal B(\mathbb R^n)\to M$ is said to be an {\it $n$-dimensional observable} if (i) $x(\mathbb R^n)=1$ and (ii) if $(A_m)_m$ is a sequence of mutually disjoint subsets from $\mathcal B(\mathbb R^n)$ and $A =\bigcup_i A_i$, then the sequence $(x(A_m))_m$ is summable and $x(A)=\sum_m x(A_m)$.}
\end{definition}

The principal properties of $n$-dimensional observables are: Let $A,B,A_i, B_i\in \mathcal B(\mathbb R^n)$, then
\begin{itemize}
\item[(i)] $x(\mathbb R^n \setminus A)=x(A)'$ and $x(\emptyset)=0$;
\item[(ii)] $x$ is monotone, i.e. $x(A)\le x(B)$ whenever $A \subseteq B$ and in such a case $x(B\setminus A)= x(B)- x(A)$;
\item[(iii)] if $(A_i)_i \searrow A$, i.e. $A_{i+1}\subseteq A_i$ and $\bigcap_i A_i=A$, then $x(A) = \bigwedge_i x(A_i)$ and if $(B_i)_i \nearrow B$, then $x(B)=\bigvee_i x(B_i)$;
\item[(iv)] $x(A)+x(B)$ exists iff does $x(A\cup B)+x(A\cap B)$, and in such a case,  $x(A)+x(B)=x(A\cup B)+x(A\cap B)$.
\end{itemize}

Many interesting examples of $n$-dimensional observables on a lexicographic MV-algebra $M$ can be obtained in the following way: Let $(a_i)_i$ be a finite or infinite sequence of summable elements of $M$ with $\sum_i a_i = 1$, and let $(\mathbf t_i)_i$ be a sequence of mutually different points of $\mathbb R^n$. Then the mapping $x: \mathcal B(\mathbb R^n) \to M$ defined by

\begin{equation}\label{eq:obser}
x(A):= \sum_{i \colon \mathbf t_i \in A} a_i, \quad A\in \mathcal B(\mathbb R^n),
\end{equation}
is an $n$-dimensional observable on $M$; we note that sum over the empty set is $0$. Lemma \ref{le:2.1}(3) shows that $x$ is defined correctly.

Now, we introduce some notations for $n$-tuples. Let $(t_1,\ldots,t_n)$ and $(s_1,\ldots,s_n)$ be two $n$-tuples from $\mathbb R^n$. We write
\begin{itemize}

\item[(i)] $(t_1,\ldots,t_n)\le(s_1,\ldots,s_n)$ if $t_i\le s_i$ for each $i=1,\ldots,n$,
\item[(ii)] $(t_1,\ldots,t_n)<(s_1,\ldots,s_n)$ if $t_i\le s_i$ for each $i=1,\ldots,n$ and for some $j\in \{1,\ldots,n\}$, we have $s_j< t_j$,
\item[(iii)] $(t_1,\ldots,t_n) \ll (s_1,\ldots,s_n)$ if $t_i< s_i$ for each $i=1,\ldots,n$.
\end{itemize}
In addition, since $M=\Gamma(H\lex G,(u,0))$, then in $M$ we can do addition as well subtraction as in the group $H\lex G$.

\begin{lemma}\label{le:3.1}
Let $x$ be an $n$-dimensional observable on a lexicographic MV-algebra $M=\Gamma(H\lex G,(u,0))$, where $(H,u)$ is a linearly ordered group and $G$ is a Dedekind $\sigma$-complete $\ell$-group.
Define a mapping $F: \mathbb R^n \to M$ by
\begin{equation}\label{eq:F}
F(t_1,\ldots,t_n):= x((-\infty,t_1,)\times \cdots\times (-\infty, t_n)),\quad t_1,\ldots,t_n \in \mathbb R.
\end{equation}
Then $F$ satisfies the following basic properties
\begin{equation}\label{eq:(3.1)}
F(s_1,\ldots,s_n)\le F(t_1,\ldots,t_n) \quad \mbox{\rm if}\quad s_i\le t_i \quad \mbox{\rm for each } i=1,\ldots,n, \quad{\rm (monotony)},
\end{equation}
\begin{equation}\label{eq:(3.2)} \bigvee_{(s_1,\ldots,s_n)}F(s_1,\ldots,s_n)=1,
\end{equation}
\begin{equation}\label{eq:(3.3)} \bigvee_{(s_1,\ldots,s_n)\ll(t_1,\ldots,t_n)}F(s_1,\ldots,s_n) = F(t_1,\ldots,t_n),
\end{equation}
\begin{equation}\label{eq:(3.4)}
\bigwedge_{t_i} F(s_1,\ldots, s_{i-1},t_i, s_{i+1},\ldots, s_n)=0 \mbox{ \rm for } i=1,\ldots,n,
\end{equation}

\begin{equation}\label{eq:(3.5)}
1\ge \Delta_{1}(a_1,b_1)\big(\cdots \big(\Delta_{n}(a_n,b_n)F(s_1,
\ldots,s_n)\big)\cdots\big)\ge 0,\quad \mbox{\rm (volume condition)},
\end{equation}
where $\Delta_{i}(a_i,b_i)H(s_1,\ldots,s_n)=H(s_1,\ldots,s_{i-1},b_i, s_{i+1},\ldots,s_n) - H(s_1,\ldots,s_{i-1},a_i, s_{i+1},\ldots,s_n)$ for  $a_i\le b_i$, $i=1,\ldots,n$, where $H: \mathbb R^n \to M$.
\end{lemma}

\begin{proof}
Monotonicity of $F$ follows from the monotonicity of $x$. The suprema and infimum on the left-hand side of (\ref{eq:(3.2)})--(\ref{eq:(3.4)}) exist thanks to monotonicity and density of rational numbers.

Volume property (\ref{eq:(3.4)}) can be obtained in the following way: We have $\emptyset = \emptyset \times (-\infty,s_2)\times \cdots \times (-\infty, s_n) = \bigcap_n ((-\infty,-n)\times (-\infty,s_2)\times \cdots \times (-\infty, s_n))$, and it implies (\ref{eq:(3.4)}).

Equality (\ref{eq:(3.5)}) follows from the following. Let $A=[a_1,b_1)\times \cdots\times [a_n,b_n)$, where $a_i\le b_i$ for each $i=1,\ldots,n$. Then $x(A)= \Delta_{1}(a_1,b_1)\cdots \Delta_{n}(a_n,b_n)F(s_1,\ldots,s_n)\ge 0$.
\end{proof}

We note that $\Delta_{1}(a_1,b_1)\big(\cdots \big(\Delta_{n}(a_n,b_n)F(s_1,
\ldots,s_n)\big)\cdots\big)$ can be read/written also as
$$
\Delta_{1}(a_1,b_1)\big(\cdots \big(\Delta_{n}(a_n,b_n)F(s_1,
\ldots,s_n)\big)\cdots\big)=\big[\cdots\big[ F(s_1,\ldots,s_n)\big]_{s_n=a_n}^{b_n}\cdots\big]_{s_1=a_1}^{b_1}.
$$

The function $F:\mathbb R^n\to M$ defined in the latter proposition is said to be an $n$-{\it dimensional spectral resolution} corresponding to $x$, we write then also $F=F_x$. In general, every mapping $F:\mathbb R^n\to M$ satisfying (\ref{eq:(3.1)})--(\ref{eq:(3.5)}) is said to be an $n$-{\it dimensional spectral resolution}. Our main task is to show when an $n$-dimensional spectral resolution $F$ implies that there is an $n$-dimensional observable $x$ on $M$ such that $F$ is an $n$-dimensional spectral resolution corresponding to $x$.

In addition, it is possible to show the following:

(1) if $\big((t^k_1,\ldots,t^k_n)\big)_k \nearrow (t_1,\ldots,t_n)$, $(t^k_1,\ldots,t^k_n) \ll (t_1,\dots,t_n)$ for $k\ge 1$, then
$$
\bigvee_{ (t^k_1,\ldots,t^k_n)\ll (t_1,\ldots,t_n)} F(t^k_1,\ldots,t^k_n) = F(t_1,\ldots, t_n).
$$

(2) If $(i_1,\ldots,i_n)$ is any permutation of $(1,\ldots,n)$, then
$$
\Delta_{1}(a_1,b_1)\big(\cdots \big(\Delta_{n}(a_n,b_n)F(s_1,\ldots,s_n)\big) \cdots\big) = \Delta_{i_1}(a_{i_1},b_{i_1})\big(\cdots\big( \Delta_{i_n}(a_{i_n},b_{i_n})F(s_1,\ldots,s_n)\big) \cdots\big).
$$
Therefore, without loss of readability, we can write
\begin{align*}
\Delta_{1}(a_1,b_1)\cdots \Delta_{n}(a_n,b_n)F(s_1,
\ldots,s_n)&=\Delta_{1}(a_1,b_1)\big(\cdots \big(\Delta_{n}(a_n,b_n)F(s_1,
\ldots,s_n)\big)\cdots\big)\\
&= \Delta_{i_1}(a_{i_1},b_{i_1})\cdots \Delta_{i_n}(a_{i_n},b_{i_n})F(s_1,\ldots,s_n).
\end{align*}

(3) If $i_1,\ldots,i_k$ are mutually different integers from $\{1,\ldots,n\}$ for $1\le k< n$, then
\begin{equation}\label{eq:(3.7)}
\Delta_{i_1}(a_{i_1},b_{i_1})\cdots \Delta_{i_k}(a_{i_k},b_{i_k})F(s_1,\ldots,s_n)\ge 0.
\end{equation}
The volume condition of the form (\ref{eq:(3.7)}) follows from the following. Let $B= B_1\times \cdots\times B_n$, where $B_i= [a_i,b_i)$ if $i\in \{i_1,\ldots,i_k\}$ and $B_i=(-\infty,s_i)$ if $i \in \{1,\ldots,n\}\setminus \{i_1,\ldots,i_k\}$. Then (\ref{eq:(3.7)}) denotes $\Delta_{i_1}(a_{i_1},b_{i_1})\cdots \Delta_{i_k}(a_{i_k},b_{i_k})F(s_1,\ldots,s_n)= x(B)\ge 0$.

If $A = [a_1,b_1)\times\cdots\times [a_n,b_n)$ with reals $a_i\le b_i$ for each $i=1,\ldots,n$, and $F$ is an $n$-dimensional spectral resolution, we define
$$
V(F,A):=\Delta_1(a_1,b_1)\cdots\Delta_n(a_n,b_n)F.
$$

As we can see, (ii) of Definition \ref{de:3.1} of an $n$-dimensional observable on $M$ deals with summable elements $(x(A_i))_i$ whenever $(A_i)_i$ is a disjoint sequence of Borel subsets of $\mathbb R^n$. Therefore, this definition is applicable also for a definition of $n$-dimensional observables on lexicographic effect algebras of the form $E=\Gamma_{ea}(H\lex G,(u,0))$, where $(H,u)$ is a unital po-group with interpolation  and $G$ is a directed Dedekind $\sigma$-complete po-group with interpolation. We note that due to \cite[Cor 2.12]{Goo}, the po-group $H \lex G$ is with interpolation. Similarly, we can define an $n$-dimensional spectral resolution on $E=\Gamma_{ea}(H\lex G,(u,0))$.

Now, we exhibit basic properties of different kinds of observables on different types of MV-algebras. The same properties will hold also for lexicographic effect algebras of the form $E=\Gamma_{ea}(H\lex G,(u,0))$.

\begin{proposition}\label{pr:3.3}
Let $x$ be a two-dimensional observable defined on a $\Rad$-Dedekind $\sigma$-complete perfect MV-algebra $M$ and let $F(s,t)=x((-\infty,s)\times (-\infty,t))$, $s,t \in \mathbb R$. Then there is a point $(s_0,t_0)\in \mathbb R$ such that $(s_0,+\infty)\times (t_0,+\infty)=\{(s,t)\in \mathbb R\colon F(s,t)\in \Rad(M)'\}$. Moreover, if $t>t_0$, then $\bigwedge_{s\searrow s_0}(F(s,t)-F(s_0,t))\in \Rad(M)'$, $\bigwedge_{s\searrow s_0} F(s,t)\in \Rad(M)'$ and if $s>s_0$, then $\bigwedge_{t\searrow t_0}(F(s,t)-F(s,t_0))\in \Rad(M)'$, $\bigwedge_{t\searrow t_0}F(s,t)\in \Rad(M)'$. In addition, the element
$$
a=\bigwedge\{F(s,t)\colon F(s,t)\in \Rad(M)'\}
$$
exists in $M$ and it belongs to $\Rad(M)'$ and
$$
x(\{(s_0,t_0)\})\in \Rad(M)'.
$$
\end{proposition}

\begin{proof}
Let $x$ be a two-dimensional observable on a perfect MV-algebra and let $F(s,t)=x((-\infty,s)\times(-\infty,t))$, $s,t\in \mathbb R$ be given. Denote by $T_1=\{(s,t)\in \mathbb R^2 \colon F(s,t) \in \Rad(M)'\}$; then $T_1\ne \emptyset$. Take $s \in \mathbb R$ such that $F(s,t)\in \Rad(M)'$ for some $t \in \mathbb R$ and let $T^s_1=\{t \in \mathbb R\colon F(s,t)\in \Rad(M)'\}$. Since $\bigwedge_t F(s,t)=0$ for each $s\in \mathbb R$, there is $t'\in \mathbb R$ such that $F(s,t')\in \Rad(M)$. Then, for each $t\in T^s_1$, we have $t'\le t$ (otherwise $t_0\le t'$ for some $t_0\in \mathbb R$ with $F(s,t_0)\in \Rad(M)'$, so that $F(s,t_0)\le F(s,t')\in \Rad(M)$, a contradiction). Then $t_s=\inf\{t\in T^s_1\}> -\infty $. According to Lemma \ref{le:2.1}, $\bigvee_{t<t_s}F(s,t)=F(s,t_s)\in \Rad(M)$, yielding $T^s_1=(t_s,+\infty)$.

We have to underline a simple but important note: If $F(s,t)\in \Rad(M)'$, then $F(s,t'), F(s',t)\in \Rad(M)'$ for all $s'>s$ and all $t'>t$.

Now, let $s_1<s_2$ be real numbers such that $F(s_1,t_1),F(s_2,t_2)\in \Rad(M)'$ for some $t_1,t_2\in \mathbb R$. Let $t'\in (t_{s_1},+\infty)$. Then $\Rad(M)'\ni F(s_1,t')\le F(s_2,t')$, so that $F(s_2,t')\in \Rad(M)'$, and $t' \in (t_{s_2},+\infty)$, $(t_{s_1},+\infty) \subseteq (t_{s_2},+\infty)$, and $t_{s_2}\le t_{s_1}$.

We assert that if $s_1<s_2$ are real numbers such that $F(s_1,t),F(s_2,t)\in \Rad(M)'$ for some $t\in \mathbb R$, then $t_{s_1}=t_{s_2}$.  Assume the converse, then $t_{s_2}<t_{s_1}$.
Then there are $t_1<t_2\in \mathbb R$ with $t_{s_2}<t_1<t_{s_1}<t_2$ such that $F(s_1,t_1)\in \Rad(M)$ and $F(s_1,t_2), F(s_2,t_1), F(s_2,t_2) \in \Rad(M)'$. Using the volume condition in the form $F(s_2,t_2)-F(s_1,t_2)\ge F(s_2,t_1)-F(s_1,t_1)$, we see that the left-hand side of the inequality is from $\Rad(M)$ whereas the right-hand one is from $\Rad(M)'$ which is a contradiction. Hence, $t_0=t_s$ for each $s\in \mathbb R$ with $F(s,t')\in \Rad(M)'$ for some $t'\in \mathbb R$ and $T^s_1=(t_0,+\infty)$.

In the same way we proceed with the second coordinate. For each $t\in \mathbb R$ such that $F(s,t)\in \Rad(M)'$ for some $s \in \mathbb R$, we define $S^t_1=\{s \in \mathbb R\colon F(s,t)\in \Rad(M)'\}$. As in the above, there is $s_t \in \mathbb R$ such that $s_t=\inf \{s \in \mathbb R\colon F(s,t)\in \Rad(M)'\}$, $\bigvee_{s<s_t} F(s,t)=F(s_t,t)\in \Rad(M)$, so that $S^t_1=(t_s,+\infty)$. Moreover, if $t_1<t_2$ and $F(s_1,t_1),F(s_2,t_2)\in \Rad(M)'$, then $s_{t_2}\le s_{t_1}$. Analogously as for $t_s$, it is possible to show that $s_{t_2}= s_{t_1}$. Whence, $s_0=s_t$ for each $t\in \mathbb R$ with $F(s',t)\in \Rad(M)'$ for some $s'\in \mathbb R$, and $S^t_1=(s_0,+\infty)$.

Taking into account that all points $t_s$ and $s_t$ lay on the half lines starting from $t_0$ and $s_0$, respectively, we conclude that $T_1 =\{(s,t)\in \mathbb R\colon F(s,t)\in \Rad(M)'\}= (s_0,+\infty)\times (t_0,+\infty)$. If we denote by $T_0=\{(s,t)\in \mathbb R^2\colon F(s,t) \in \Rad(M)\}$, then $T_0=\mathbb R^2\setminus T_1$.

We note that if $s \in \mathbb R$ is such $F(s,t')\in M_0$ for some $t'\in \mathbb R$, let $T_0^s=\{t\in \mathbb R\colon F(s,t)\in \Rad(M)\}$. Then either $T_0^s = (-\infty,t_s]$ or $T_0^s = (-\infty,+\infty)$. If we define dually $S_0^t=\{s\in \mathbb R\colon F(s,t)\in \Rad(M)\}$, then either $S_0^t=(-\infty,s_t]$ or $S_0=(-\infty,+\infty)$.

Let $(s',t')\in \mathbb R^2$ be given. Since $\{(s',t')\}=\bigcap_{(s,t)\gg(s',t')} \Big(\big((-\infty,s)\times (-\infty,t)\setminus (-\infty,s')\times (-\infty,t)\big)\setminus\big((-\infty,s)\times (-\infty,t')\setminus(-\infty,s')\times (-\infty,t')\big)\Big)$, we have

\begin{equation}\label{eq:x(s,t)}
x(\{s',t'\})=\bigwedge_{(s,t)\gg (s',t')} \Delta_1(s',s)\Delta_2(t',t) F(s,t).
\end{equation}

Now, take $(s_0,t_0)\in \mathbb R^2$. Using (\ref{eq:x(s,t)}), we have $x(\{s_0,t_0\})=\bigwedge_{(s,t)\gg (s_0,t_0)}\big((F(s,t)-F(s_0,t))-(F(s,t_0)-F(s_0,t_0))\big)\in \Rad(M)'$.
\end{proof}

The point $(s_0,t_0)$ is said to be a {\it characteristic point} of $F$.

\begin{proposition}\label{pr:3.4}
Let $x$ be an $n$-dimensional observable on a $\Rad$-Dedekind $\sigma$-complete perfect MV-algebra $M$ and let $F(t_1,\ldots,t_n)=x((-\infty,t_1)\times \cdots\times (-\infty,t_n))$, $t_1,\ldots,t_n\in \mathbb R$. Then there is a point $(t^0_1,\ldots,t^0_n)\in \mathbb R^n$, called a characteristic point of $F$, such that $(t^0_1,+\infty)\times \cdots\times (t^0_n,+\infty)=\{(t_1,\ldots,t_n)\in \mathbb R^n\colon F(t_1,\ldots,t_n)\in \Rad(M)'\}$. Moreover, for each $i=1,\ldots,n$ and each $t_i>t^0_i$, the element
$$
\bigwedge_{t_i\searrow t^0_i} (F(t_1,\ldots,t_i,\ldots,t_n)-F(t_1,\ldots,t^0_i,\ldots,t_n))
$$
exists in $M$ and belongs to $\Rad(M)'$, the element

\begin{equation}\label{eq:Rad}
\bigwedge_{t_i\searrow t^0_i} F(t_1,\ldots,t_i,\ldots,t_n)
\end{equation}
exists in $M$ and belongs to $\Rad(M)'$,
the element
\begin{equation}\label{eq:Radn}
a=\bigwedge\{F(t_1,\ldots,t_n)\colon F(t_1,\ldots,t_n)\in \Rad(M)'\}
\end{equation}
exists also in $M$ and it belongs to $\Rad(M)'$, and
$$
x(\{(t^0_1,\ldots,t^0_n)\})=\bigwedge_{(t_1,\ldots,t_n)\gg (t^0_1,\ldots,t^0_n)}\Delta_1(t^0_1,t_1)\cdots\Delta_n(t^0_n,t_n) F(s_1,\ldots,s_n)\in \Rad(M)'.
$$
\end{proposition}

\begin{proof}
If $n=1$, the statement follows from \cite[Prop 4.5]{DDL}, if $n$=2, we have Proposition \ref{pr:3.3}. Thus let $n>2$.
Denote by $T_1=\{(t_1,\ldots,t_n)\in \mathbb R^n\colon F(t_1,\ldots,t_n)\in \Rad(M)'\}$ and let $j\in \{1,\ldots,n\}$ be fixed. Without loss of generality we will assume that $j=1$. Let $t_2,\ldots,t_n \in \mathbb R$ be fixed real numbers such that $F(t,t_2,\ldots,t_n)\in \Rad(M)'$ for some $t \in \mathbb R$. As in the proof of Proposition \ref{pr:3.3}, we can show that there is $t^0_1= t^0_1(t_2,\ldots,t_n)=\inf\{t \in \mathbb R\colon F(t,t_2,\ldots,t_n)\in \Rad(M)'\}$. In addition, due to \cite[Lem 2.2]{DvLa}, we have $\bigvee_{t_1<t^0_1}F(t_1,\ldots,t_n)=F(t^0_1,t_2,\ldots,t_n)\in \Rad(M)$. Therefore, $(t^0_1,+\infty)
=\{t\in \mathbb R\colon F(t,t_2,\ldots,t_n)\in \Rad(M)'\}$.

Let $i\in \{2,\ldots,n\}$. It is evident that if $t'_i<t''_i$, then $t^0_1(t_2,\ldots,t''_i,\ldots,t_n)\le  t^0_1(t_2,\ldots,t'_i,\ldots,t_n)$. We assert that
\begin{equation}\label{eq:t}
t^0_1(t_2,\ldots,t''_i,\ldots,t_n)= t^0_1(t_2,\ldots,t'_i,\ldots,t_n)
\end{equation}
whenever $t'_i,t''_i>t^0_i$.
Indeed, if not, there are real numbers $t_1'<t_1''$ with $t^0_1(t_2,\ldots,t''_i,\ldots,t_n)< t'_i< t^0_1(t_2,\ldots,t'_i,\ldots,t_n) <t''_i$ such that
\begin{align*}
&F(t'_1,t_2,\ldots,t_i',\ldots,t_n)\in \Rad(M)\\
\end{align*}
and
\begin{align*}
F(t''_1,t_2,\ldots,t_i',\ldots,t_n), F(t'_1,t_2,\ldots,t_i'',\ldots,t_n), F(t''_1,t_2,\ldots,t_i'',\ldots,t_n)\in \Rad(M)'.
\end{align*}
The volume condition is reducing to
$$
F(t''_1,t_2,\ldots,t_i'',\ldots,t_n)- F(t'_1,t_2,\ldots,t_i'',\ldots,t_n)\ge F(t''_1,t_2,\ldots,t_i',\ldots,t_n)- F(t'_1,t_2,\ldots,t_i',\ldots,t_n).$$
The left-hand side of this volume condition belongs to $\Rad(M)$ but the right-hand side is from $\Rad(M)'$ which is a contradiction. So that (\ref{eq:t}) holds.  Since (\ref{eq:t}) holds for each $i=2,\ldots,n$, we have for all $(t_2',\ldots,t'_n),(t_2'',\ldots,t''_n)\in \mathbb R^{n-1}$
$$
t^0_1:=t^0_1(t_2',\ldots,t'_i,\ldots,t'_n)= t^0_1(t_2'',\ldots,t''_i,\ldots,t''_n).
$$
The above reasoning holds for each $j\in \{1,\ldots,n\}$, so that there is
$t^0_j=t^0_j(t_1,\ldots,t_{j-1},t_{j+1},\ldots,t_n)$ such that we have
$T_1=(t^0_1,+\infty)\times \cdots\times (t^0_n,+\infty)$.

The properties (\ref{eq:Rad})--(\ref{eq:Radn}) and the last one follow the same reasonings as those in the proof of the analogous statement in Proposition \ref{pr:3.3} and using the equality
$$
x(\{(t^0_1,\ldots,t^0_n)\})= \bigwedge_{(t_1,\ldots,t_n)\gg (t^0_1,\ldots,t^0_n)} \Delta_1(t^0_1,t_1)\cdots\Delta_n(t^0_n,t_n)F(s_1,\ldots,s_n).
$$
\end{proof}

\section{Characteristic Points}\label{sec:4}

In the section, we will study characteristic points of observables/spectral resolutions in the case of $k$-perfect MV-algebras where $k>1$. We show that in this case the situation is more complicated than for perfect MV-algebras because it can happen that some $T_i=\{(s_1,\ldots,s_n)\in \mathbb R^n \colon F(s_1,\ldots,s_n)\in M_i\}$ has more characteristic points.

Let $x$ be an $n$-dimensional observable on a $\Rad$-Dedekind $\sigma$-complete $k$-perfect MV-algebra $M$ and let $F$ be an $n$-dimensional spectral resolution on $M$.
We define $T_i=\{(s_1,
\ldots,s_n)\in \mathbb R^n \colon F(s_1,\ldots,s_n)\in M_i\}$ for $i=0,1,\ldots,k$. Let $i=1,\ldots,k$ be fixed and given $\mathbf s=(s_1,\ldots,s_n)\in T_i$, we define
$$
s^i_j= \inf\{s\in \mathbb R\colon (s_1,\ldots,s_{j-1},s,s_{j+1},\ldots,s_n)\in T_i\}
$$
for each $j=1,\ldots,n$. We write also $s^i_j=\pi^i_j(\mathbf s)$ as the $j$-th projection of $\mathbf s$ in $T_i$.
Then the point $(s_1,\ldots,s_{j-1}, s^i_j,s_{j+1},\ldots,s_n)\in T_l$ for some $l=0,1,\ldots,i-1$.
The point $\pi_i(\mathbf s):=(s^i_1,\ldots,s^i_n)$ is said to be a {\it characteristic point associated} to $\mathbf s \in T_i$, or simply a {\it characteristic point} of $T_i$, and we have $\pi_i(\mathbf s)\in T_l$ for unique $l=0,1,\ldots,i-1$. Sometimes we say also that it is a characteristic point of $F$ not emphasizing $T_i$.
We note that it can happen that some $T_i$ has more characteristic points, see Figure 1. However, using the proof of Proposition \ref{pr:3.4}, we have that $T_k$ has always a unique characteristic point. Trivially, if $n=1$, then every non-empty $T_i$ has a unique characteristic point. The aim of the section is to show when non-empty $T_i$ has finitely many characteristic points. The complete situation is described for $n=2$. We start with a useful definition and a simple observation.

We define a relation $\approx_i$ on $T_i\ne \emptyset$, $i>0$, as follows: Let $\mathbf s=(s_1,\ldots,s_n), \mathbf t=(t_1,\ldots,t_n)\in T_i$. We write
$$
\mathbf s\approx_i \mathbf t \quad \Leftrightarrow \quad \pi_i(\mathbf s)=\pi_i(\mathbf t).
$$
Then clearly $\approx_i$ is an equivalence and let $\mathbf s/\approx_i=\{\mathbf t \in T_i \colon \mathbf t \approx_i \mathbf s\}$, $\mathbf s \in T_i$. The set $\mathbf s/\approx_i$ is said to be a {\it block} in $T_i$, or simply a block. 
If $\mathbf s_1,\mathbf s_2 \in \mathbf s/\approx_i$, we have $\pi_i(\mathbf s_1)=\pi_i(\mathbf s_2)$, so that each block has a unique characteristic point.

\begin{lemma}\label{le:approx1}
Let $x$ be an $n$-dimensional observable and $F=F_x$.
Let $i$ be the least non-zero integer such that $T_i$ is non-empty. Then $T_i$ has only finitely many characteristic points.
\end{lemma}

\begin{proof}
Let $\mathbf s=(s_1,\ldots,s_n)\in T_i$ be given and let $\mathbf t=(t_1,\ldots,t_n)$ be a characteristic point associated to $\mathbf s\in T_i$. Define a semi-closed rectangle $A(\mathbf t,\mathbf s)=[ t_1,s_1)\times \cdots \times [ t_n,s_n)$. Then every vertex of $A(\mathbf t,\mathbf s)$ besides of $\mathbf s\in T_i$ is from $T_0$. Hence, $x(A(\mathbf t,\mathbf s))=\Delta_1(t_1,s_1)\cdots\Delta_n(t_n,s_n)F \in T_i$. We note that if $\mathbf t_1$ and $\mathbf t_2$ are two different characteristic points of $T_i$ corresponding to $\mathbf s_1,\mathbf s_2\in T_i$, then $A(\mathbf t_1,\mathbf s_1)\cap A(\mathbf t_2, \mathbf s_2)=\emptyset$. Because in $M_i$ there are only finitely many mutually orthogonal elements, and $x(A(\mathbf t,\mathbf s))$'s are mutually orthogonal for different characteristic points $\mathbf t$'s, we have that $T_i$ has only finitely many characteristic points.
\end{proof}

\begin{lemma}\label{le:approx2}
Let $x$ be an $n$-dimensional observable and $F=F_x$.
Let $\mathbf t=(t_1,\ldots,t_n)\in \mathbb R^n$ and $\mathbf s\in T_i$ be such that $\pi_i(\mathbf s)\ll \mathbf t \le \mathbf s$. Then $\mathbf t \in T_i$, $\mathbf t \approx_i \mathbf s$ and
\begin{equation}\label{eq:approx1}
\pi_i(\mathbf s)=\bigwedge \{\mathbf t\in \mathbb R^n\colon \pi_i(\mathbf s)\ll \mathbf t \le \mathbf s\}
\end{equation}
and the element
\begin{equation}\label{eq:approx2}
\bigwedge\{F(\mathbf t)\colon \pi_i(\mathbf s)\ll \mathbf t \le \mathbf s\}
\end{equation}
exists in $M$ and belongs to $M_i$.
\end{lemma}

\begin{proof}
For simplicity, we denote $\mathbf t=(0,\ldots,0)$ and $\mathbf s=(1,\ldots,1)$, and we create an $n$-dimensional cube with all vertices $(i_1,\ldots,i_n)$, where each $i_j \in \{0,1\}$. We denote by $\mathbf e_j$ an $n$-dimensional vector consisting of $1$'s only and at the $j$-th coordinate we have $0$. Then $F(\mathbf e_j)\in M_i$ for each $j=1,\ldots,n$.

First, using the mathematical induction with respect to $n$, we prove that $F(0,\ldots,0)\in M_i$. If $n=1$, the statement is trivial. Let $n=2$. Then $F(1,1),F(1,0),F(0,1)\in M_i$ and $F(0,0)\in M_{i'}$ where $i'\le i$. Using the volume condition, we have $M_0 \ni F(1,1)-F(1,0)\ge F(0,1)-F(0,0)\in M_{i-i'}$, so that $i=i'$. Therefore, $F(0,0)\in M_i$.

Let $n=3$. Then using the volume condition of the form (\ref{eq:(3.7)}) to points $(0,0,1),(1,0,1),(0,1,1),(1,1,1)$, we obtain $M_0\ni F(1,1,1)-F(1,0,1)\ge F(0,1,1)-F(0,0,1)\in M_{i-i'}$ and $i=i'$, where $F(0,0,1)\in M_{i'}$, so that $F(0,0,1)\in M_i$. Repeating this ideas to all faces of the cube, we obtain finally that $F(0,\ldots,0)\in M_i$.

Assume that the statement holds for each $m\le n$ and we prove it for $n+1$. Take all $n+1$-dimensional vectors from $\{0,1\}^{n+1}$ whose first coordinate is $1$. It forms an $n$-dimensional subcube of our cube and by hypothesis, for all its vertices $\mathbf f$ we have $F(\mathbf f)\in M_i$. Fixing the second coordinate, we get that for all its vertices $\mathbf g$, we have $F(\mathbf g)\in M_i$. Finally, we obtain that also an $n$-dimensional subcube containing $(0,\ldots,0)$ has the property that $(0,\ldots,0)\in T_i$.

Summarizing, we see that (\ref{eq:approx1}) holds and since $x$ is an observable, from Lemma \ref{le:2.1}(2) we conclude the element in (\ref{eq:approx2}) exists and belongs to $M_i$ because
$$
x\big(\bigcap_{\pi_i(\mathbf s)\ll \mathbf t \le \mathbf s} (-\infty,t_1)\times\cdots\times (-\infty,t_n)\big)=\bigwedge\{F(\mathbf t)\colon \pi_i(\mathbf s)\ll \mathbf t \le \mathbf s\}.
$$

Clearly, each point $\mathbf s(t_j):= (s_1,\ldots,s_{j-1},t_j,s_{j+1},\ldots,s_n)$ belongs to $T_i$ for each $j=1,\ldots,n$, and according to proof of Proposition \ref{pr:3.3}, we have
$\pi^i_{j'}(\mathbf s(t_j))=\pi^i_{j'}(\mathbf s)$ for each $j'=1,\ldots,n$, so that $\mathbf s(t_j)\approx_i \mathbf s$ which entails $\mathbf t \approx_i \mathbf s$.
\end{proof}

We note that the element in (\ref{eq:approx2}) can be calculated also in the form
\begin{equation}\label{eq:approx3}
\bigwedge\{F(\mathbf t)\colon \pi_i(\mathbf s)\ll \mathbf t \le \mathbf s\}
= \bigwedge\{F(\mathbf t)\colon \mathbf t \in \mathbf s/\approx_i\}.
\end{equation}
Moreover, the elements of the latter equation exist in $M$, they belong to $M_i$, and they will play an important role by definition of an $n$-dimensional spectral resolution for lexicographic MV-algebras and lexicographic effect algebras. Moreover, as it was shown in \cite{DDL, DvLa1}, if $n=1$, then there are spectral resolutions which do not correspond to any observable with for which (\ref{eq:approx3}) does not exist in $M$. The existence of all elements (\ref{eq:approx3}) will be crucial for showing when an $n$-dimensional spectral resolution $F$ implies the existence of an $n$-dimensional observable $x$ such that $F=F_x$.

In addition, if $n=2$ then for $(s_1,t_1), (s_2,t_2)\in T_i$, we have if $(s_1,t_1)\not\approx_i (s_2,t_2)$, then $s_1\ne s_2$ and $t_1\ne t_2$. In addition, $s_1<s_2$ iff $t_1>t_2$.

The following simple lemma is of a remarkable importance which will be used later.

\begin{lemma}\label{le:rays}
Let $F$ be a two-dimensional spectral resolution on a $\Rad$-Dedekind $\sigma$-complete $k$-perfect MV-algebra and let $(s_0,t_0)$ be a characteristic point of $F$.
Given a vertical ray $\{(s_0,t)\colon t>t_0\}$ and two points $(s_1,t)\in T_m$, $(s_2,t)\in T_j$ such that $s_1\leq s_0<s_2$ and $t_0<t$, then $m<j$. Analogously for horizontal rays.
\end{lemma}

\begin{proof}
Assume $s_0,s_1,s_2$ and $t$ are the reals from the statement. As $(s_0,t_0)$ is a characteristic point, it is associated to some $(s_3,t_3)\in T_i$. If $t_0<t\leq t_3$, then $(s_2,t)$ belongs to $T_j$ for $j\geq i$. But $(s_1,t)\leq (s_0,t)\leq(s_0,t_3)$ and $(s_0,t_3)$ belongs to $T_m$ with $m<i$.

If $t_3<t$, then the desired claim follows from the volume condition as
$F(s_2,t)-F(s_1,t)\geq F(s_2,t_3)-F(s_1,t_3)$ because if $F(s_2,t)\in M_j$, $F(s_1,t)\in M_m$, $F(s_2,t_3)\in M_{j'}$, and  $F(s_1,t_3)\in M_{m'}$ for some $j'$ such that $j\ge j'\ge i$ and $m'$ such $m'\le m$, we have $F(s_2,t)-F(s_1,t)\in M_{j-m}$ and $F(s_2,t_3)-F(s_1,t_3)\in M_{j'-m'}$ so that $j-m\ge j'-m' \ge i- m'>0$ and $j>m$.
\end{proof}

\begin{theorem}\label{th:approx4}
Let $F$ be a two-dimensional spectral resolution on a $\Rad$-Dedekind $\sigma$-complete $k$-perfect MV-algebra. Then every $T_i\ne \emptyset$, $i\ne 0$, has at most $k$ characteristic points. There could be at most $k^2$ characteristic points of $F$.
\end{theorem}

\begin{proof}
Assume $(s_1,t_1),\ldots,(s_m,t_m)$ is an arbitrary finite system of mutually different characteristic points of $F$. Without loss of generality (see the note just before Lemma \ref{le:rays}) we can assume that $s_1<\cdots<s_m$; if some $s_i=s_j$, then $t_i\ne t_j$, so they will be distinguished in the second coordinate.
Now, if we take $t$ greater than all the $t_i$'s, we find some $s'_i$, $i=0,\ldots, m$ such that $s'_{i-1}<s_i<s'_i$. Applying Lemma \ref{le:rays}, we get a sequence of points $(s'_i,t)$'s, all from different $T_i$'s. Hence $m\leq k$. If we make a similar deduction for the second coordinate, we realize the statement is true. Hence, the number of characteristic points of $F$ is at most $k^2$.

Now, let $(s_1,t_1),\ldots,(s_m,t_m)$ be any finite system of mutually different characteristic points of $T_i$. If, e.g. $(s_1,t_1)\ne (s_2,t_2)$ are two characteristic points of $T_i$, then, say $s_1<s_2$, which implies $t_2<t_1$. Therefore, without loss of generality, we can assume that $s_1<\cdots<s_m$ and $t_1>\cdots> t_m$. Repeating the proof of the latter paragraph, we have $m\le k$.
\end{proof}

A strengthening of the latter result will be done in Theorem \ref{th:approx-num'} at the end of this section.

In the following simple examples, using observables $x$ of the form (\ref{eq:obser}), we illustrate different situations with $T_i$'s and characteristic points for $n=2$.

\begin{example}\label{ex:3.7}
Let $M=\Gamma(\mathbb Z\lex \mathbb Z,(2,0))$ and $a_1=(0,1)$, $a_2=(1,2)$, $a_3=(1,-3)$.

\begin{itemize}
\item[{\rm (1)}] Let $t_1=(1,1)$, $t_2=(2,2)$, $t_3=(3,3)$. Then $T_2= (2,+\infty)\times (2,+\infty)$, $T_1 = (2,+\infty)\times (2,+\infty)\setminus (1,+\infty)\times (1,+\infty)$, and $T_0=
    \mathbb R^2 \setminus (1,+\infty)\times (1,+\infty)$. Characteristic points are $(2,2)$ and $(3,3)$.

\item[{\rm (2)}] Let $t_1=(1,1)$, $t_2=(2,2)$, $t_3=(3,2)$. Then
$T_2 = (3,+\infty)\times (2,+\infty)$, $T_1= (2,3]\times (2,+\infty)$, $T_0= \mathbb R^2\setminus (2,+\infty)\times (2,+\infty)$. Characteristic points are $(2,2)$ and $(3,2)$.

\item[{\rm (3)}] Let $t_1=(1,1)$, $t_2=(2,2)$, $t_3=(2,3)$. Then $T_2=(2,+\infty)\times (3,+\infty)$, $T_1= (2,+\infty)\times (2,+\infty)\setminus (2,+\infty)\times (3,+\infty)$, $T_0= \mathbb R^2 \setminus (2,+\infty)\times (2,+\infty)$. Characteristic points are $(2,2)$ and $(2,3)$.

\item[{\rm (4)}] Let $t_1=(1,1)$, $t_2= (3,3)$, $t_3=(4,2)$. Then $T_2= (4,+\infty)\times (2,+\infty)$, $T_1= (3,4]\times (3,+\infty)$ and $T_0=\mathbb R^2 \setminus (T_1\cup T_2)$. Characteristic points are $(3,3)$ and $(4,2)$.

\item[{\rm (5)}] Let $a_1=(0,2)$ and $a_2 = (2,-2)$. If $t_1=(1,1)$ and $t_2 = (2,2)$, then $T_2= (2,+\infty)\times (2,+\infty)$, $T_1=\emptyset$ and $T_0=\mathbb R^2 \setminus T_1$. The only characteristic point is $(2,2)$.

\item[{\rm (6)}] Let $M=\Gamma(\mathbb Z\lex \mathbb Z,(3,0))$ and $a_1=(1,1)$, $a_2=(1,2)$, $a_3=(1,-3)$ and $t_1=(1,2)$, $t_2=(2,1)$, $t_3=(3,3)$. Then $T_3 = (3,+\infty)\times (3,+\infty)$, $T_2 = (2,+\infty)\times (1,+\infty)\setminus T_3$, $T_1= (1,2]\times (2,+\infty)$, $T_0=(1,+\infty)\times (2,+\infty)\setminus (T_1\cup T_2\cup T_3)$ and the characteristic points are $(1,2)$, $(2,1)$, $(3,3)$.

\item[{\rm (7)}] Let $M=\Gamma(\mathbb Z\lex \mathbb Z,(3,0))$ and $a_1=(1,1)$, $a_2=(1,2)$, $a_3=(1,-3)$ and $t_1=(1,3)$, $t_2=(2,2)$, $t_3=(3,1)$. Then $T_3 = (3,+\infty)\times (3,+\infty)$, $T_2= (2,3]\times (2,+\infty) \cup (3,+\infty)\times (2,3]$, $T_1 = (1,2] \cup (3,+\infty) \cup (2,3] \times (2,3] \cup (3,+\infty) \times (1,2]$, $T_0= \mathbb R^2 \setminus (T_1\cup T_2 \cup T_3)$, and the characteristic points are $(1,3)$, $(2,2)$, $(3,1)$, $(2,3)$, $(3,2)$ and $(3,3)$. See Figure 1.

\item[{\rm (8)}] Let $M=\Gamma(\mathbb Z\lex \mathbb Z,(3,0))$ and $a_1=(1,1)$, $a_2=(2,-1)$, and $t_1=(1,2)$, $t_2=(2,1)$. Then $T_3 = (2,+\infty)\times (2,+\infty)$, $T_2=(2,+\infty)\times (1,2]$, $T_1=(1,2] \times (2,+\infty)$, $T_3=\mathbb R^2 \setminus (T_1\cup T_2\cup T_3)$ and the characteristic points are $(1,2),(2,1),(2,2)$.

\item[{\rm (9)}] Let $M=\Gamma(\mathbb Z\lex \mathbb Z,(2,0))$ and $a_1=(1,2)$, $a_2=(1,1)$, $a_3=(0,-3)$ and $t_1=(1,3)$, $t_2=(3,2)$, $t_3=(2,0)$. Then $T_1=(1,3] \times (3,+\infty)$, $T_2=(3,+\infty)\times (2,3)$, $T_3=(3,+\infty)\times (3,+\infty)$, $T_0=\mathbb R^2 \setminus (T_1\cup T_2\cup T_3)$, and the characteristic points are $(1,3),(3,2), (3,3)$.
\end{itemize}
\end{example}

\begin{figure}
\begin{center}
\unitlength 1.00mm
\linethickness{0.4pt}
\begin{picture}(28.00,30.44)
\put(-10.00,9.00){\line(0,1){25.00}}
\put(-10.00,9.00){\line(1,0){50.00}}
\put(0.00,25.00){\line(0,1){5.00}}
\put(0.00,25.00){\line(1,0){35.00}}
\put(10.00,20.00){\line(0,1){10.00}}
\put(10.00,20.00){\line(1,0){25.00}}
\put(20.00,15.00){\line(0,1){15.00}}
\put(20.00,15.00){\line(1,0){15.00}}
\put(0.00,25.00){\line(0,1){5.00}}
\put(0.00,15.00){$T_0$}
\put(3.00,26.00){$T^1_1$}
\put(13.00,26.00){$T^1_2$}
\put(23.00,26.00){$T_3$}
\put(13.00,21.00){$T^2_1$}
\put(23.00,21.00){$T^2_2$}
\put(23.00,16.00){$T^3_1$}
\put(0.00,25.00){\circle*{1.00}}
\put(10.00,20.00){\circle*{1.00}}
\put(20.00,15.00){\circle*{1.00}}
\put(20.00,20.00){\circle*{1.00}}
\put(20.00,25.00){\circle*{1.00}}
\put(10.00,25.00){\circle*{1.00}}
\put(-12.00,29.00){$t$}
\put(35.00,6.00){$s$}
\put(-10.00,-2.00){$T_1=T^1_1\cup T^2_1\cup T^3_1$, $T_2=T^1_2\cup T^2_2$}
\end{picture}
\end{center}
\caption{Example \ref{ex:3.7}(7)\label{fano'}}
\end{figure}

In the following result, we show how it is with characteristic points of $F$, where $F(s_1,\ldots,s_n)=x((-\infty,s_1)\times\cdots\times (-\infty,s_n))$, $s_1,\ldots,s_n\in \mathbb R$ and $x$ is an $n$-dimensional observable on a lexicographic MV-algebra $M=\Gamma(H\lex G,(u,0))$, where $(H,u)$ is a linearly ordered unital group and $G$ is a Dedekind $\sigma$-complete $\ell$-group. We recall $[0,u]_H:=\{h \in H\colon 0\le h \le u\}$ and  similarly we define $[0,u)_H$ or $(0,u)_H$ and $(0,u]_H$, etc. For any $h \in [0,u]_H$, let $M_h=\{(a,b)\in \Gamma(H\lex G,(u,0))\colon a=h\}$, and
let $T_h=\{(s_1,\ldots,s_n)\in \mathbb R^n\colon F(s_1,\ldots,s_n)\in M_h\}$. Then $T_0$ and $T_u$ are non-void. Let $h \in (0,u)_H$ be such that $T_h \ne \emptyset$ be fixed.  For every $\mathbf s = (s_1,\ldots,s_n)\in T_h$, we define projections
$$
s^h_j= \inf\{s\in \mathbb R\colon (s_1,\ldots,s_{j-1},s,s_{j+1},\ldots,s_n)\in T_h\}
$$
and we write $s^h_j=\pi_j^h(\mathbf s)$. Then the point $(s_1,\ldots,s^h_j,\ldots,s_n)\in T_{h'}$ for some $h'\in [0,h)_H$ and
for each $j=1,\ldots,n$. The point $\pi_h(\mathbf s)=(\pi^h_1(\mathbf s),\ldots,\pi^h_n(\mathbf s))$ is said to be a {\it characteristic point} of $F$ in $T_h$ or, simply a characteristic point of $F$. As in the proof of Proposition \ref{pr:3.4}, we can show that $T_u$ has a unique characteristic point.

For each $h \in (0,u]_H$ with $T_h$ non-void, we define a relation $\approx_h$ as follows: For $\mathbf s=(s_1,\ldots,s_n)\in T_h$ and
$\mathbf t=(t_1,\ldots,t_n)\in T_h$, we write $\mathbf s \approx_h \mathbf t$ iff $\pi^h_i(\mathbf s)=\pi^h_i(\mathbf t)$ for each $i=1,\ldots,n$. Then $\approx_h$ is an equivalence and let $\mathbf s/\approx_h =\{\mathbf t \colon \mathbf t \approx_h \mathbf s\}$; it is called also a {\it block} (of $T_h$). Due to definition of a block, it is clear that every block has a unique characteristic point. It is possible to show that Lemmas/Theorem \ref{le:approx2}--\ref{th:approx4} hold also for this case of general lexicographic MV-algebras.

It is necessary to underline that the characteristic points and blocks can be defined in the same way also for general $n$-dimensional spectral resolutions (not only for those $F$ which are associated to some $n$-dimensional observable $x$, i.e. $F=F_x$). Also in such a case, every block has a unique characteristic point.

Now, we show that if $x$ is a two-dimensional observable on $M=\Gamma(H\lex G,(u,0))$, then $F=F_x$ has only finitely many characteristic points.

\begin{theorem}\label{th:approx6}
Suppose that $x$ is a two-dimensional observable on an MV-algebra $M=\Gamma(H\lex G,(u,0))$, where $(H,u)$ is a unital linearly ordered
group and $G$ is a Dedekind $\sigma$-complete $\ell$-group.
Then $F(s,t)=x((-\infty,s)\times (-\infty,t))$, $s,t \in \mathbb R$, is a two-dimensional spectral resolution with finitely many characteristic points.
\end{theorem}

\begin{proof}
Similarly as in the proof of Proposition \ref{pr:3.4}, $T_u$ has a unique characteristic point $(s_u,t_u)$, so that $(s,t)\in T_u$ iff $(s_u,t_u)\ll(s,t)$. Let $(s,t)$ be a characteristic point of some $T_h$ for some $h\in (0,u)_H$. Then $(s_u,t_u)\not\ll (s,t)$.
Applying Lemma \ref{le:rays} to the characteristic points $(s,t)$ and $(s_u,t_u)\in T_u$, we obtain that $(s,t)<(s_u,t_u)$, that is, all characteristic points of $F$ are below $(s_u,t_u)$.


We know, that if we fix any $s'$, then $x_{s'}$, where $x_{s'}(A)=x((-\infty,s)\times A)$, $A \in \mathcal B(\mathbb R)$, can be viewed as a one-dimensional observable on the MV-algebra $[0,u_s]\subseteq M$, where $u_s=x((-\infty,s)\times \mathbb R)$. Using \cite[Thm 4.2]{DvLa1},
then $F(s',t)$, $t \in \mathbb R$, meets only finitely many $T_h$'s. An analogous result holds for a fixed $t'$.

In particular this holds for any $s' >s_u$ as well as for any $t'>t_u$ and let $m_{s'}$ and $m_{t'}$ be the number of these intersections for a fixed $s'>s_u$ and fixed $t'>0$, respectively. Let $\mathbf s_1=(s_1,t_1),\ldots,\mathbf s_k=(s_k,t_k)$ be any finite system of mutually different characteristic points of $F$. Let $t_{i_1}>\cdots> t_{i_l}$ be all those $t_i$'s that are mutually different.
In view of Lemma \ref{le:rays}, there are only finitely many reals, which occur as the first coordinate of some characteristic point from $\mathbf s_1,\ldots,\mathbf s_k$. If some $t_i=t_j$, they can be distinguished in the first coordinate. So we have less than $m_{s'}$ characteristic points which have different $t_i$'s.
Repeating this procedure for the second coordinate and $t'>t_u$, we have that also in this case we have only finitely many characteristic points with different first coordinates.
Hence, we conclude that $k\le m_{s'}\cdot m_{t'}$, and there are only finitely many characteristic points.
\end{proof}

In \cite{DvLa1}, it was shown that even for a one-dimensional mapping $F: \mathbb R\to M$ which satisfies (\ref{eq:(3.1)})--(\ref{eq:(3.5)}) it can happen that it has infinitely many characteristic points.
For a general two-dimensional spectral resolution, that is a mapping $F: \mathbb R^2 \to M$ satisfying only (\ref{eq:(3.1)})--(\ref{eq:(3.5)}), a situation like the one sketched in the Figure~\ref{Fig 2} could occur. So there could be infinitely many characteristic points. Therefore, such $F$ cannot be extended to a two-dimensional observable.

\begin{figure}
\begin{center}
\unitlength 1.00mm
\linethickness{0.4pt}
\begin{picture}(28.00,30.44)
\put(-10.00,0.00){\line(0,1){25.00}}
\put(-10.00,0.00){\line(1,0){50.00}}

\put(-10.00,0.00){\circle*{1.00}}

\put(6.00,16.00){\line(0,1){9.00}}
\put(6.00,16.00){\line(1,0){34.00}}

\put(6.00,16.00){\circle*{1.00}}

\put(-2.00,8.00){\line(0,1){17.00}}
\put(-2.00,8.00){\line(1,0){42.00}}

\put(-2.00,8.00){\circle*{1.00}}

\put(-6.00,4.00){\line(0,1){21.00}}
\put(-6.00,4.00){\line(1,0){46.00}}

\put(-6.00,4.00){\circle*{1.00}}

\put(-8.00,2.00){\line(0,1){23.00}}
\put(-8.00,2.00){\line(1,0){48.00}}

\put(-8.00,2.00){\circle*{1.00}}

\put(-9.00,1.00){\line(0,1){24.00}}
\put(-9.00,1.00){\line(1,0){49.00}}

\put(-9.00,1.00){\circle*{1.00}}
\end{picture}
\end{center}
\caption{Example \label{Fig 2}}
\end{figure}

However some finiteness property needs to hold:

\begin{lemma}\label{le:approx7}
Let $F:\mathbb R^2\to M$ be a mapping satisfying {\rm (\ref{eq:(3.1)})--(\ref{eq:(3.5)})}.
Then there is no infinite bounded antichain of characteristic points.
\end{lemma}

\begin{proof} We will use the following well-known Monotone Subsequent Theorem (MST), see e.g. \cite[Thm 3.4.7]{BaSh}
saying: For each real sequence we can find a monotone subsequence.

Suppose $\big((s_n,t_n)\big)_n$ is a bounded antichain. By (MST), there is an increasing sequence $(n_i)_i$ of integers such that $(s_{n_i})_i$ is a monotone sequence. We can use (MST) again to find a subsequence $(n_{i_j})_j$ of the sequence $(n_i)_i$ such that both $(s_{n_{i_j}})_j$ and $(t_{n_{i_j}})_j$ are monotone as well as bounded. Hence, we may assume without lost of generality that the coordinates $s_n$'s and $t_n$'s of our original antichain both satisfy the monotony property.

As $\big((s_n,t_n)\big)_n$ is an antichain, both the first and the second coordinates are strictly monotone sequences and they are of opposite directions. From symmetry we may assume $(s_n)_n$ is (strictly) increasing and $(t_n)_n$ is (strictly) decreasing. But then the points $F(s_n,t_1)$, $ n\ge 1$, all lie in different $T_h$'s, see Lemma~\ref{le:rays}. It follows $\bigvee_n F(s_n,t_1)$ cannot exist. However, the definition of spectral resolution demands to exist and to equal $F(s_0,t_1)$, where $s_0=\sup_n s_n$, so we have a contradiction.
\end{proof}

However, the previous Lemma is the best we can hope. The schema of characteristic points in Figure \ref{Fig 3} demonstrates a situation, where lengths of antichains are unbounded: In the pictured situation all the characteristic points are placed in a union of an infinite chain of squares and as the squares are getting smaller, they contain an antichain of higher length.

\begin{figure}
\centering
\caption{Spectral resolution with unbounded antichains of characteristic points}
\label{Fig 3}
\end{figure}

\begin{theorem}\label{th:approx8}
Let $F:\mathbb R^2\to M$ be a mapping satisfying {\rm (\ref{eq:(3.1)})--(\ref{eq:(3.5)})}. Then $F$ has only countable many characteristic points.
\end{theorem}

\begin{proof}
We can repeat the proof of Theorem~\ref{th:approx4}, the only change is that in the last argument we deduce that there are at most countably many reals that occur as the first (respectively second) coordinate of some characteristic point. We use here a fact that, for each fixed $s_0\in\mathbb{R}$, $\{F(s_0,t)\colon t\in\mathbb{R}\}$ meets only countably many $M_h$'s. Indeed, given $s_0,t_0\in\mathbb{R}$, we have $F(s_0,t_0)=\bigvee_{t<t_0} F(s_0,t)$. So there has to be some $t_1$, such that all $F(s_0,t)$'s, $t_1<t<t_0$, live in one $M_h$ (for some $h\in H$). If $\{F(s_0,t)\colon t\in\mathbb{R}\}$ meets uncountably many different $M_h$'s, we would have an uncountable collection of disjoint non-empty open intervals in $\mathbb{R}$. But that is a contradiction, as $\mathbb{R}$ has a countable dense subset $\mathbb{Q}$.
\end{proof}

As we have seen, a necessary condition for a two-dimensional spectral resolution $F$ on a lexicographic MV-algebra $\Gamma(H\lex G,(u,0))$ is that $F$ has finitely many characteristic points. Therefore, we say that an $n$-dimensional spectral resolution $F$ on $\Gamma(H\lex G,(u,0))$ has the {\it finiteness property} if $F$ has only finitely many characteristic points. A special attention will be paid for those $F$ for which every non-empty $T_h$ has a unique characteristic point.

As we have seen in the proof of Theorem \ref{th:approx6}, every characteristic point of $F$ is under a unique characteristic point of $T_u$. We note that if $h_1< h_2$, $h_1,h_2 \in (0,u)_H$, and if $(s_{h_i},t_{h_1}), (s_{h_2},t_{h_2})$ are characteristic points of $T_{h_1}$ and $T_{h_2}$, respectively, then it does not mean $(s_{h_i},t_{h_1})\le (s_{h_2},t_{h_2})$, in general, even if every non-empty $T_h$ has a unique characteristic point.  Indeed, see Example \ref{ex:3.7}(8) and characteristic points $(1,2)$ and $(2,1)$.

For the next proposition, we introduce the following notion. We say that a block $B$ of $T_h$ of an $n$-dimensional spectral resolution $F$ on $\Gamma(H\lex G,(u,0))$ is $T_0$-{\it adjoined} if, for some $\mathbf s =(s_1,\ldots,s_n) \in B$ (and therefore for each $\mathbf s \in B$), we have that $(s_1,\ldots,s_{j-1},\pi^h_j(\mathbf s),s_{j+1},\ldots,s_n)\in T_0$ for each $j=1,\ldots,n$.

\begin{proposition}\label{pr:approx}
Let $F$ be an $n$-dimensional spectral resolution on a $\Rad$-Dedekind $\sigma$-complete $k$-perfect MV-algebra $M$ such that there are finitely many $T_0$-adjoint blocks $B_1,\ldots,B_m$ with the property that
$$
a_j=\bigwedge\{F(s_1,\ldots,s_n)\colon (s_1,\ldots,s_n)\in B_j\}
$$
exists in $M$ for each $j=1,\ldots,m$ and $\sum_{j=1}^ma_j=1$, and let $T_0=\{(s_1,\ldots,s_n)\in \mathbb R^n\colon F(s_1,\ldots,s_n)=0\}$. Then $F$ has finitely many characteristic points, and there is an $n$-dimensional observable $x$ on $M$ such that $F=F_x$.
\end{proposition}

\begin{proof}
Let $\mathbf s_j$ be a unique characteristic point of a $T_0$-adjoint block $B_j$ for each $j=1,\ldots,m$. If we take these characteristic points and the elements $a_1,\ldots,a_m$ from the conditions of the Proposition, the mapping $x: \mathcal B(\mathbb R^n)\to M$ defined by (\ref{eq:obser}) is an $n$-dimensional observable on $M$ and $F=F_x$. Therefore, each non-empty $T_i$ has only finitely many characteristic points.
\end{proof}

For example, such a two-dimensional spectral resolution satisfying the latter proposition is described in Example \ref{ex:3.7}(7). On the other hand, in Example \ref{ex:3.7}(9), we have two blocks $B_1$ and $B_2$ such that $a'_1=(1,2)$ and $a'_2=(1,-2)$, where $a'_1$ and $a'_2$ are elements calculated by Proposition \ref{pr:approx}. Then $a'_1+a'_2=1$, but the formula (\ref{eq:obser}) defines an observable $x'$ such that $F\ne F_{x'}$ because $T_0\ne \{(s,t)\in \mathbb R^2 \colon F(s,t)=0\}$.

\begin{theorem}\label{th:approx-num'}
Let $F$ be a two-dimensional spectral resolution on a $\Rad$-Dedekind $\sigma$-complete $k$-perfect MV-algebra $M$.
Then every non-empty $T_i$ has at most $k-i+1$ characteristic points and whence, $F$ has at most $k(k+1)/2$ characteristic points.
\end{theorem}

\begin{proof}
(1) Define by $\mathcal P(\mathbb  R^2)$ the system of finite unions of semi-closed rectangles in $\mathbb R^2$. Then every element of $\mathcal P(\mathbb R^2)$ can be expressed as a finite union of mutually disjoint semi-closed rectangles. If $A=[ a_1,b_1)\times [ a_2,b_2)$ and let $c\in \mathbb R$ lie between $a_1$ and $b_1$, then $V(F,A)= V(F,A_1)+V(F,A_2)$, where $A_1 =[ a_1,c)\times [ a_2,b_2)$ and $A_2=[ a,b_1)\times [ a_2,b_2)$. Dually it holds for a point $d\in \mathbb R$ between $b_1$ and $b_2$. If we define $x_F(C)= \sum_{l=1}^n V(F,A_l)$, where $A=\bigcup_{l=1}^n A_l\in \mathcal P(\mathbb R^2)$ and $A_1,\ldots,A_n$ are mutually disjoint semi-closed rectangles, then it is possible to show that $x_F$ is correctly defined an additive mapping on $\mathcal P(\mathbb R^2)$. Moreover, if $C\in \mathcal P(\mathbb R^2)$ and $C \subseteq (-\infty,s)\times (-\infty,t)$, then $x_F(C)\le F(s,t)$; the latter property holds due to fact that if $D$ is the least semi-closed rectangle containing $C$, then $x_F(C)\le x_F(D)\le F(s,t)$. For more details on this construction, see \cite[Sec 8]{Hal}.

(2) According to Theorem \ref{th:approx4}, we know $F$ has finitely many characteristic points. Let $T_i\ne \emptyset$ for $i=1,\ldots,k$ and let $B_1,\ldots,B_m$ be all blocks of $T_i$. Let $(s_1,t_1)\in B_1,\ldots,(s_m,t_m)\in B_m$ be a fixed finite system of points from $T_i$.
Since for $(s_1,t_1)$ and $(s_2,t_2)$ we have either $s_1<s_2$ and $t_1>t_2$ or $s_1>s_2$ and $t_1<t_2$, without loss of generality, we can
assume that $s_1<\cdots<s_m$ and $t_1>\cdots> t_m$.

Define $s'_1$ to be a real number small enough, so that $(s'_1,t_1)\in T_0$ and for each $j$, $2\leq j\leq m$ define $s'_j:=\pi^i_1(s_j,t_j)$. Then $(s'_j,t_j)\in T_{j'}$, where $j'<j$. Next find a real $t$ small enough, so that $(s_m,t)\in T_0$ (and consequently $(s_j,t)\in T_0$ for each $j=1,\ldots,m$, by monotony). Now define for each $j=1,\ldots,m$ a rectangle $C_j=[ s'_j,s_j)\times [ t,t_j)$. The rectangles are pairwise disjoint ($s_j\leq s'_{j+1}$ for $j=1\ldots,m-1$) and due to the volume condition we can associate to them the area elements $x_1=V(F,C_1)\in M_i$, $x_j=V(F,C_2)\in M_{j-j'},\ldots, x_m=V(F,C_m)\in M_{m-m'}$.
We see that $x_1\in M_i$ and $x_j\notin M_0$ for $j=2,\ldots,m$. By part (1) of the present proof, we see that the elements $x_1,\ldots,x_n$ are summable. Therefore, we have $i+(m-1)\le i+\sum_{j=2}^m (j-j')\le k$,
$i+(m-1)\leq k$, equivalently $m\leq k-i+1$.

Consequently, the number of characteristic points of $F$ is at most $k+(k-1)+\cdots + 1=(k+1)k/2$.
\end{proof}

The bound from the last theorem is the best one: In the general situation, let $a_1,\ldots,a_k\in M_1$ be summable elements such that $a_1+\cdots+a_k=1$. Choose $(s_1,t_1),\ldots,(s_k,t_k)$ mutually different points in $\mathbb R^2$ and define an observable $x$ by (\ref{eq:obser}). Then for $F=F_x$, we have that every $T_i$, $i=1,\ldots,k$, has exactly $k-i+1$ characteristic points, and $F$ has $k(k+1)/2$ characteristic points.

\end{document}